\documentclass[11pt]{article}
\usepackage[dvips]{epsfig}
\usepackage{psfrag}
\usepackage{amsmath}
\usepackage{amssymb}
\usepackage{amscd}
\usepackage{picinpar}
\usepackage{times}
\usepackage{pb-diagram}
\usepackage{graphicx}
\usepackage{wrapfig}
\usepackage{xspace}

\newtheorem{theorem}{Theorem}[section]
\newtheorem{lemma}[theorem]{Lemma}
\newtheorem{proposition}[theorem]{Proposition}
\newtheorem{corollary}[theorem]{Corollary}
\newtheorem{definition}[theorem]{Definition}

\newenvironment{proof}[1][Proof]{\begin{trivlist}
\item[\hskip \labelsep {\bfseries #1}]}{\end{trivlist}}

\newcommand{\qed}{\nobreak \ifvmode \relax \else
      \ifdim\lastskip<1.5em \hskip-\lastskip
      \hskip1.5em plus0em minus0.5em \fi \nobreak
      \vrule height0.75em width0.5em depth0.25em\fi}

\begin{document}

\title{Variational Principles  for Minkowski Type Problems,
Discrete Optimal Transport, and Discrete Monge-Ampere Equations}

\author{
Xianfeng Gu, Feng Luo, Jian Sun, S.-T. Yau }

\date{}
\maketitle
\begin{abstract}
In this paper, we develop several related finite dimensional
variational principles for discrete optimal transport (DOT),
Minkowski type problems for convex polytopes and discrete
Monge-Ampere equation (DMAE). A link between the discrete optimal
transport, discrete Monge-Ampere equation and the power diagram in
computational geometry is established.
\end{abstract}

\section{Introduction}
\subsection{Statement of results}
The classical Minkowski problem for convex body has influenced the
development of  convex geometry and differential geometry through
out the twentieth century. In its simplest form, it states,

\medskip

\begin{figure}[h]
\centering
\includegraphics[width=0.2\textwidth]{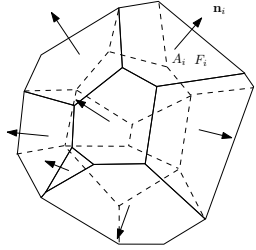}
\caption{Minkowski problem
\label{fig:minkowski_problem}}
\end{figure}

\noindent {\bf Problem 1 (Minkowski problem for compact polytopes
in $\bf R^n$)} \it Suppose $n_1, ..., n_k$ are unit vectors which
span $\bf R^n$ and $A_1, ..., A_k>0$ so that $\sum_{i=1}^k A_i
n_i=0$. Find a compact convex polytope $P \subset \bold R^n$ with
exactly $k$ codimension-1 faces $F_1, ..., F_k$ so that $n_i$ is
normal to $F_i$ and the area of $F_i$ is $A_i$. \rm

\medskip

Minkowski's famous solution to the problem says that the polytope
$P$ exists and is unique up to parallel translation. Furthermore,
Minkowski's proof is variational and suggests an algorithm to find
the polytope.

Minkowski problem for unbounded convex polytopes was solved by
Alexandrov in his influential book on convex polyhedra \cite{alex}.  In particular, he proved the following fundamental
theorem (Theorem 7.3.2) which is one of the main focus of our
investigation.

\begin{figure}[h]
\centering
\includegraphics[width=0.6\textwidth]{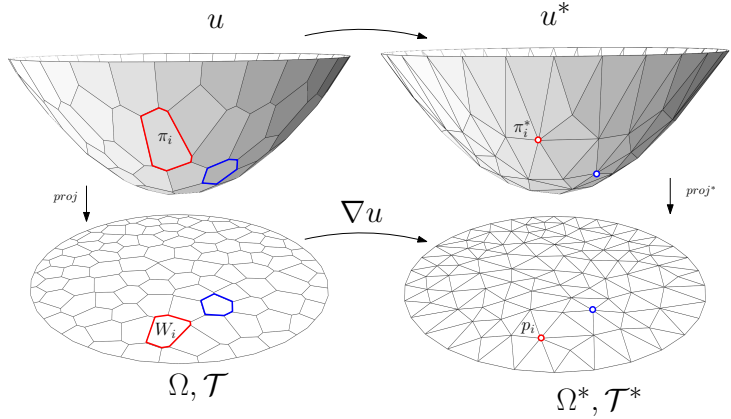}
\caption{Discrete Optimal Transport Mapping (left to right): map $W_i$ to $p_i$. 
Discrete Monge-Ampere equation (right to left): $vol(W_i)$ is the discrete Hessian determinant of $p_i$.
\label{fig:discrete_VP_framework}}
\end{figure}

\begin{theorem} (Alexandrov)
Suppose $\Omega$ is a compact convex polytope with non-empty
interior in $\bf R^n$, $p_1, ..., p_k \subset \bf R^n$ are
distinct $k$ points and $A_1, ..., A_k >0$ so that
$\sum_{i=1}^kA_i=vol(\Omega)$. Then there exists a vector $h=(h_1,
..., h_k) \in \bf R^k$, unique up to adding the constant
$(c,c,..., c)$,  so that the piecewise linear convex function
$$ u(x) = \max_{ x \in \Omega} \{ x \cdot p_i + h_i \}$$ satisfies
$vol(\{ x \in \Omega | \nabla u(x) = p_i\}) =A_i$.
\label{thm:alexandrov}
\end{theorem}

The functions $u$ and $\bigtriangledown u(x)$ in the theorem will
be called the \it Alexandrov potential \rm and  \it Alexandrov
map.
 \rm Alexandrov's proof is non-variational and non-constructive.
Producing a variational proof of it was clearly in his mind.
Indeed, on page 321 of ~\cite{alex}, he asked if one can find a
variational proof and considered such proof ``is of prime
importance by itself". One of the main results of the paper gives
a (finite dimensional) variational proof Alexandrov's
Theorem~\ref{thm:alexandrov}. Indeed, we give a variational proof
of a general version of Theorem~\ref{thm:alexandrov}
(Theorem~\ref{thm:otp} below) and produce an algorithm for finding
the function $u$. In recent surge of study on optimal transport,
Theorem~\ref{thm:alexandrov} is reproved and is a very special case of the seminal
work of Brenier (see for instance \cite{vil}, Theorem 2.12(ii),
and Theorem 2.32). Brenier proved that the function
$\bigtriangledown u$  minimizes the quadratic cost $\int_{\Omega}
| x - T(x)|^2 dx$ among all measure preserving maps (transport
maps) $T: (\Omega, dx) \to (\bf R^n$, $\sum_{i=1}^k A_i
\delta_{p_i}$). Here $\delta_p$ is the Dirac measure supported at
the point $p$.
Thus our work produces a variational principle
and an algorithm for finding Alexandrov maps with finite images.

\subsection{Variational principles}
Here is a simple framework which we will use to establish
variational principles for solving equations in this paper.
Suppose $X \subset \bf R^k$ is a simply connected open set and
$A(x)=(A_1(x), ..., A_k(x)): X \to \bf R^k$ is a smooth function
so that $\frac{\partial A_i(x)}{\partial x_j}=\frac{\partial
A_j(x)}{\partial x_i}$ for all $i,j$.  Then for any given $B=(B_1,
..., B_k) \in \bf R^k$, solutions $x$ of the equation $A(x) =B$
are exactly the critical points of the function $E(x) =\int_a^x
\sum_{i=1}^k ( A_i(x) -B_i) dx_i$. Indeed, the assumption
$\frac{\partial A_i(x)}{\partial x_j}=\frac{\partial
A_j(x)}{\partial x_i}$ says the differential 1-form $\omega
=\sum_{i=1}^k (A_i(x)-B_i)dx_i$ is closed in the simply connected
domain $X$. Therefore the integral $E(x) =\int^x_a \omega$ is well
defined independently of the choice of the path from $a$ to $x$.
By definition, $\frac{\partial E(x)}{\partial x_i} = A_i(x)-B_i$,
i.e., $\bigtriangledown E(x) =A(x)-B$. Thus $A(x)=B$ is the same
as $\bigtriangledown E(x)=0$.

All variational principles established in this paper use the above
framework. We will use the above framework to give a variational
proof of Alexandrov's theorem.

The paper will mainly deal with piecewise linear (PL) convex
functions. Here are the notations. Given $p_1, ..., p_k \in \bf
R^n$ and $h=(h_1, ..., h_k) \in \bf R^k$, we use $u(x) = u_{h}(x)$
to denote the PL convex function
$$ u_h(x) =\max_{i} \{ x \cdot p_i + h_i\},$$
where $ u \cdot v$ is the dot product. Let $W_i(h) =\{ x \in \bf
R^n
$$| \bigtriangledown u(x) = p_i\}$ be the closed convex polytope.
It is well known that $W_i(h)$ may be empty or unbounded.  One of
the main result we will prove is,

\begin{theorem}  Let $\Omega$ be a compact convex domain in $\bf R^n$
and $\{p_1, ..., p_k\}$ a set of distinct points in $\bf R^n$ and
$\sigma: \Omega \to \bf R$ be a positive continuous function. Then
for any $A_1, ..., A_k >0$ with $\sum_{i=1}^k A_i =\int_{\Omega}
\sigma(x) dx$, there exists $b=(b_1, ..., b_k) \in \bf R^k$,
unique up to adding a constant $(c,..., c)$, so that $\int_{
W_i(b) \cap \Omega}\sigma(x) dx  =A_i$ for all $i$. The vectors
$b$ are exactly maximum points of the convex function 
\begin{equation}
E(h) = \int^h_a \sum_{i=1}^k \int_{W_i(h) \cap \Omega}\sigma(x) dx  dh_i -\sum_{i=1}^k h_i A_i
\label{eqn:energyE}
\end{equation}
on the open convex set $H =\{ h \in \bf
R^k$ $| vol(W_i(h) \cap \Omega)
>0$ for all $i$\}. In fact, $E(h)$ restricted to $H_0= H
\cap \{ h | \sum_{i=1}^k h_i =0\}$ is strictly convex. Furthermore, 
$\nabla u_b$ minimizes the quadratic cost $\int_{\Omega}
| x - T(x)|^2 \sigma dx$ among all transport
maps $T: (\Omega, \sigma dx) \to (\bf R^n$, $\sum_{i=1}^k A_i \delta_{p_i})$. 
\label{thm:otp}
\end{theorem}

We remark that Alexandrov's theorem corresponds to $\sigma \equiv 1$.
The existence and the uniqueness of Theorem~\ref{thm:otp} are special
case of important work of Brenier on optimal transport. 
Our main contribution is the variational formulation. 
The Hessian of the function $E(h)$ has a clear geometric meaning
and is easy to compute (see equation~\ref{eq:hessian}), which enables
one to efficiently compute Alexandrov map using Newton's method. 
Furthermore, as a consequence of our proof, we obtain a new proof of 
the infinitesimal rigidity theorem of Alexandrov that 
$ \nabla E: H_0 \to W=\{ (A_1, ..., A_k) \in {\bf R^k} | A_i > 0, \sum_{i=1}^k A_i =\int_{\Omega}
\sigma(x) dx\} $  
is a local diffeomorphism (see Corollary~\ref{cor:inft}).
We remark that Aurenhammer et al. \cite{aure1} also noticed the convexity of the function $E$, 
and they gave an elegant and simple proof of  $\bigtriangledown u_b$ minimizing 
quadratic cost.

\subsection{Discrete Monge-Ampere equation (DMAE)}
Closely related to the optimal transport problem is the
Monge-Ampere equation (MAE). Let $\Omega$ be a compact domain in
$\bf R^n$, $g: \partial \Omega \to \bf R $ and $A: \Omega \times \bf
R \times \bf R^n \to \bf R$ be given. Then the Dirichlet problem
for MAE is to find a function $w: \Omega \to \bf R$ so that

\begin{equation}
\begin{cases} det (Hess(w))(x)  =  A(x, w(x), \bigtriangledown w(x)) \\  w|_{\partial \Omega}   =  g  \end{cases}
\end{equation}

There are vast literature and deep results known on the existence,
uniqueness and regularity of the solution of MAE. We are
interested in solving the discrete version of MAE in the simplest
setting where $A(x, w, \bigtriangledown w) =A(x): \Omega \to \bf
R$ so that $A(\Omega)$ is a finite set.   By taking
Fenchel-Lengendre dual of the Alexandrov potential function $u$,
we produce a finite dimensional variational principle for solving
a discrete Monge-Ampere equation.

In the discrete setting, one of the main tasks is to define the
discrete Hessian determinant for piecewise linear function.  We
define,

\begin{definition} Suppose $(X, \mathcal T)$ is a domain in $\bf
R^n$ with a convex cell decomposition $\mathcal T$ and $w: X \to
\bf R$ is a convex function which is linear on each cell  (a PL
convex function). Then the discrete Hessian determinant of $w$
assigns each vertex $v$ of $\mathcal T$ the volume of the convex hull
of the gradients of the $w$ at top-dimensional cells adjacent to $v$.
\end{definition}

One can define the discrete Hessian determinant of any piecewise
linear function by using the signed volumes. This will not be
discussed here.
With the above definition of discrete Hessian determinant,
following Pogorelov \cite{pogo}, one formulates the Dirichlet
problem for discrete MAE (DMAE) as follows.
\vspace{4mm}

\noindent {\bf Problem 2 (Dirichlet problem for discrete MAE
(DMAE))} {\it Suppose $\Omega =conv(v_1, ..., v_m)$ is an
n-dimensional compact convex polytope in $\bf R^n$ so that  $v_i
\notin conv(v_1, ..., v_{i-1}, v_{i+1}, ..., v_k)$ for all $i$.
Let $p_1, ..., p_k$ be in $int(\Omega)$. Given any $g_1, ..., g_m
\in \bf R$ and $A_1, ..., A_k >0$, find a convex subdivision
$\mathcal T$ of $\Omega$ with vertices exactly $\{ v_1, ..., v_m,
p_1, ..., p_k\}$ and a PL convex function $w: \Omega \to \bf R$
which is linear on each cell of $\mathcal T$ so that

(a) (Discrete Monge-Ampere Equation) the discrete Hessian
determinant of $w$ at $p_i$ is $A_i$,

(b) (Dirichlet condition)  $w(v_i) = g_i$.
\rm}

\medskip

In \cite{pogo}, Pogorelov solved the above problem affirmatively.
He showed that the PL function $w$ exists and is unique. However,
his proof is non-variational. We improve Pogorelov's theorem to
the following.


\begin{theorem}
Suppose $\Omega =conv(v_1, ..., v_m)$ is an n-dimensional
compact convex polytope in $\bf R^n$ so that  $v_i \notin
conv(v_1, ..., v_{i-1}, v_{i+1}, ..., v_k)$ for all $i$ and $p_1,
..., p_k $ are in the interior of $\Omega$.  For any $g_1, ...,
g_k \in \bf R$ and $A_1, ..., A_k
>0$, there exists convex cell decomposition $\mathcal T$ having
$v_i$ and $p_j$ as vertices and a piecewise linear convex function
$w: (\Omega, \mathcal T) \to \bf R$ so that $w(v_i) = g_i,  i=1,..., m$ and the
discrete Hessian determinant of $w$ at $p_j$ is $A_j$, $j=1,..., k$.
In fact, the solution $w$ is the Legendre
dual of $\max\{ x \cdot p_j + h_j, x \cdot v_i - g_i | j=1,..., k,
i=1,..., m\}$ and $h$ is the unique critical point of a strictly
convex function.
\label{thm:dmae}
\end{theorem}


The paper is organized as follows. In \S2, we recall briefly some
basic properties of piecewise linear convex functions, their dual
and power diagrams. Theorems~\ref{thm:otp} and~\ref{thm:dmae} are proved in \S3 and
\S4.

\section{Preliminary on PL convex functions, their duals and power diagrams}

We collect some well known facts about PL convex functions, their
Legendre-Fenchel duals and their relations to power diagrams in
this sections. Most of the proofs are omitted.  See Aurenhammer~\cite{Amgous},
Passera and Rullg{\aa}rd~\cite{passare}, Siersmas and van Manen~\cite{Siada} and others for details.

The following notations will be used. For $u, v, p_1, ..., p_k \in
\bf R^n$, we use $u \cdot v$ to denote the dot product of $u, v $
and $conv(p_1, ..., p_k)$ to denote the convex hull of $\{p_1,
..., p_k\} \subset \bf R^n$.  A convex polyhedron is the
intersection of finitely many closed half spaces. A convex
polytope is the convex hull of a finite set. The relative interior
of a compact convex set $X$ will be denoted by $int(X)$.

\subsection{Legendre-Fenchel dual and PL convex functions}

The \it domain \rm of a function $f: \bf R^n \to$$ (-\infty,
\infty]$, denoted by $D(f)$, is the set $\{ x \in \bf R^n|$$f(x) <
\infty\}$.  A function $f$ is called \it proper \rm if $D(f) \neq
\emptyset$.  For a proper function $f: \bf R^n \to (-\infty,
\infty]$, the \it Legendre-Fenchel duality \rm (or simply the
dual) of $f$ is the proper function  $f^*: \bf R^n \to $$(-\infty,
\infty]$ defined by
$$ f^*(y) =\sup\{ x \cdot y - f(x) | x \in  \bf R^n\}.$$  It is well known that $f^*$ is a proper,
lower semi continuous convex function.  For instance, for the
linear function $f(x) = a \cdot x+b$, its dual $f^*$ has domain
$D(f^*)=\{a\}$ so that $f^*(a)=-b$. The Legendre-Fenchel duality
theorem says that for a proper lower semi continuous convex
function $f$, $(f^*)^* =f$.

For $P=\{p_1, ..., p_k\} \subset \bf R^n$ and $h=(h_1, ..., h_k)
\in \bf R^k$, we define the piecewise linear (PL) convex function
$u_h(x)$ to be

\begin{equation}\label{eq:PL}
u(x) =(u_h(x)=u_{h, P}(x)) =\max\{ p_i \cdot x + h_i | i=1,...,
k\}
\end{equation} The domain $D(u^*)$ of the dual $u^*$ is the convex
hull $conv(p_1, ..., p_k)$ so that
\begin{equation}\label{eq:dual}
 u^*(y) = \min\{ -\sum_{i=1}^k t_i h_i |  t_i \geq 0, \sum_{i=1}^k
 t_i =1, \sum_{t=1}^k t_i p_i=y\}
\end{equation}
(See theorem 2.2.7 of H{\"o}rmander's book on Notions of Convexity
\cite{hor}). In particular, $u^*$ is PL convex in the domain
$D(u^*)$. For instance if $h=0$, then $u^*(y)=0$ on $D(u^*)$.
Another useful consequence is,

\begin{corollary} \label{coro:11}  If $p_i \notin conv(p_1, ..., p_{i-1}, p_{i+1}, ..., p_k)$,
then
\begin{equation}\label{eq:dualvalue}
u^*(p_i) = -h_i.
\end{equation}
\end{corollary}
 Indeed,  the only way to express $p_i$ as a convex
combination of $p_1, ..., p_k$ is $p_i = 1 \cdot p_i$. Thus
(\ref{eq:dualvalue}) holds.

\begin{figure}[!t]
\centering
\includegraphics[width=0.4\textwidth]{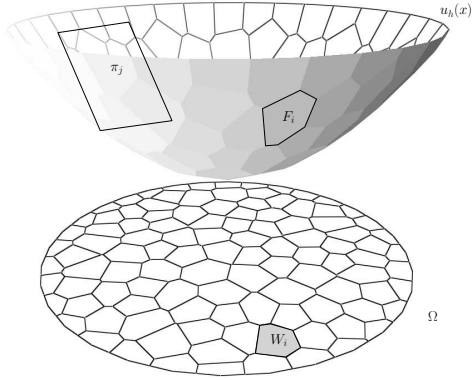}
\caption{PL-convex function and its induced convex subdivision.
\label{fig:PL_convex_function}}
\end{figure}

\subsection{PL convex functions, convex subdivisions and power diagrams}

A PL convex function $f$ defined on a closed convex polyhedron $K$
 produces a convex subdivision (called natural subdivision)
$\mathcal T$ of $K$. It is the same as the power diagram used in
computational geometry. Let us recall briefly the definition (see
for instance \cite{passare}). A \it convex subdivision \rm of $K$
is a collection $\mathcal T$ of  convex polyhedra (called cells)
so that (a) $K =\cup_{ \sigma \in \mathcal T} \sigma$, (b) if
$\sigma, \tau \in \mathcal T$, then $\sigma \cap \tau \in \mathcal
T$, and (c) if $\sigma \in \mathcal T$ and $\tau \subset \sigma$,
then $\tau \in \mathcal T$ if and only if $\tau$ is a face of
$\sigma$.  The collection $\mathcal T$ is determined by its
top-dimensional cells. The set of all zero-dimensional cells in
$\mathcal T$, denoted by $\mathcal T^0$, is called the vertices of
$\mathcal T$.

If $f$ is a PL convex function  defined on a convex polyhedron
$K$, the natural convex subdivision $\mathcal T$ of $K$ associated
to $f$ is the subdivision whose top-dimensional cells in $\mathcal
T$ are the largest convex subsets on which $f$ are linear. The \it
vertices \rm of $f$ are defined to be the vertices of $\mathcal
T$. Suppose $\{v_1, ..., v_m\}$ is the set of all vertices of $f$.
Then $f$ is determined by its vertices $\{v_i\}$ and the values at
the vertices $\{f(v_i)\}$. Indeed the graph of $f$ over $K$ is the
lower boundary of the convex hull $conv((v_1, f(v_1)), ..., (v_,
f(v_m)))$.   Recall that if $P$ is a convex polyhedron in $\bf
R^{n} \times \bf R$, then the \it lower faces \rm of $P$ are those
faces $F$ of $P$ so that if $x \in F$, then $x-(0,..., 0,
\lambda)$ is not in $P$ for all $\lambda > 0$.  The lower boundary
of $P$ is the union of all
 lower faces of $P$.

One can also describe $\mathcal T$ by using the epigraph. The
epigraph  $\{ (x, t) \in K \times \bf R$$| t \geq f(x)\}$ of $f$
is naturally a convex polyhedron. Each cell in $\mathcal T$ is the
vertical projection of a lower face of the epigraph.

Since the dual function $f^*$ is also PL convex on its domain
$D(f^*)$, there is the associated convex subdivision $\mathcal
T^*$ of $D(f^*)$. These two subdivisions $(D(f), \mathcal T)$ and
$(D(f^*), \mathcal T^*)$ are dual to each other in the sense that
there exists a bijective map $\mathcal T \to \mathcal T^*$ denoted
by $\sigma \to \sigma^*$ so that (a) $\sigma, \tau \in \mathcal T$
with $\tau \subset \sigma$ if and only if $\sigma^* \subset
\tau^*$ and (b) if $\tau \subset \sigma$ in $\mathcal T$, then the
$cone(\tau, \sigma)$ is dual to $cone(\sigma^*, \tau^*)$. Here the
cone $cone(\tau, \sigma) =\{ t(x-y) | x \in \sigma, y \in \tau, t
\geq 0\}$ and dual of a cone $C$ is $\{ x \in \bf R^n|$ $y \cdot x
\leq 0$ for all $y \in C$\}. See proposition 1 in section 2 of
Passera and Rullg{\aa}rd~\cite{passare}.

For the PL convex function $u_h(x)$ given by (\ref{eq:PL}), define
the convex polyhedron $W_i = W_i(h) =\{ x \in \bf R^n |$ $ x \cdot
p_i + h_i \geq x \cdot p_j + h_j$ for all $ j$\}. (Note that $W_i$
may be the empty set.) By definition, the convex subdivision
$\mathcal T$ of $\bf R^n$ associated to $u_h$ is the union of all
$W_i$'s and their faces.  Identity (\ref{eq:dual}) for $u^*$ says
that the graph $\{(y, u^*(y)) | y \in conv(p_1, ..., p_k)\}$ of
$u^*$  is the lower boundary of the convex hull $conv((p_1, -h_1),
..., (p_k, -h_l))$.

We summarize the convex subdivisions associated to PL convex
functions as follows,

\begin{proposition}\label{prop:21} (a) If $int(W_i(h)) \neq \emptyset$ and $p_1, ..., p_k$ are distinct, then
 $int(W_i(h)) =\{ x \in
\bf R^n |$$ x \cdot p_i + h_i > \max_{j \neq i}\{ x \cdot p_j +
h_j\}\}$.

(b) If $p_i \notin conv(p_1, ..., p_{i-1}, p_{i+1}, ..., p_k)$,
then $int(W_i(h)) \neq \emptyset$ and $W_i(h)$ is unbounded.

(c)  If $conv( p_1, ..., p_{i-1}, p_{i+1}, ..., p_k)$ is
$n$-dimensional and $p_i \in int(conv( p_1, ..., p_{i-1}, p_{i+1},
..., p_k))$, then  $W_i(h)$ is either bounded or empty.

(d) If $p_1, ..., p_k$ are distinct so that $int(W_i(h)) \neq
\emptyset$ for all $i$, then the top-dimensional cells of
$\mathcal T$ associated to $u_h$ are exactly $\{ W_i (h) |
i=1,...,k\}$, vertices of $u^*$ and the dual subdivision $\mathcal
T^*$ of $conv(p_1, ..., p_k)$ are exactly $\{p_1, ..., p_k\}$.

(e) For any distinct $p_1, ..., p_k \in \bf R^n$, there is $h\in
\bf R^k$ so that $int(W_i(h)) \neq \emptyset$ for all $i$.

\end{proposition}

\begin{proof}
 To see (a), by definition $\{ x \in \bf R^n |$$ x
\cdot p_i + h_i > \max_{j \neq i}\{ x \cdot p_j + h_j\}\}$ is open
and is in $W_i(h)$. Hence it is included in the interior
$int(W_i(h))$. Let $L_j(x) = x \cdot p_j +h_j$.  By definition,
$L_i(x) \geq L_j(x)$ for all $x \in W_i(h)$. It remains to show
for each $p \in int(W_i)$, $L_i(p)
> L_j(p)$ for $j \neq i$.  Take a point $q \in int(W_i(h))$ so
that $L_i(q)> L_j(q)$ (this is possible since $L_i \neq L_j$ for
$j \neq i$). Choose a line segment  $I$ from $q$ to $r$ in
$int(W_i(h))$ so that $p \in int(I)$. Then using $L_i(q)>L_j(q),
L_i(r) \geq L_j(r)$, $p = tq+(1-t)r$ for some $t \in (0,1)$ and
linearity, we see that $L_i(p)> L_j(p)$. This establishes (a).

To see part (b) that $int(W_i) \neq \emptyset$,  by the identity
(\ref{eq:dual}) for $u^*_{h, P'}$ with $P'=\{p_1, ..., p_k\}
-\{p_i\}$, we have $u^*_{h, P'}(p_i) =\infty$.  But $u^*_{h,
P'}(p_i) =\sup\{ x \cdot p_i -u_{h, P'}(x) | x \in \bf R^n\}$.
Hence there exists $x$ so that $x \cdot p_i + h_i > u_{h, P'}(x) =
\max_{j \neq i} (x \cdot p_j + h_j)$, i.e., $int(W_i(h)) \neq
\emptyset$. Furthermore, $u^*_{h,P'}(p_i)=\infty$ implies $W_i(h)$
is non-compact, i.e., unbounded.

To see (c),  suppose otherwise, then the set $W_i(h)$
 contains a ray $\{ t v + a | t \geq 0\}$ for some non-zero vector
 $v$. Therefore, $(t v +a) \cdot p_i + h_i \geq (tv+a) \cdot p_j +
 h_j$ for all $j \neq i$. Divide the inequality by $t$ and let $t
 \to \infty$, we obtain $v \cdot p_i \geq v \cdot p_j$ for all $j
 \neq i$. This shows that the projection of $p_i$ to the line $\{ t v | t \in \bf
 R\}$ is not in the interior of the convex hull of the
 projections of $\{p_1, ..., p_k\} -\{p_i\}$. This contradicts the
 assumption that $p_i$ is in the interior of the n-dimensional convex hull.

The first part of (d) follows from the definition. The duality
theorem (proposition 1 in section 2 of  ~\cite{passare}) shows the
second part.

To see part (e), let us relabel the set $p_1, ..., p_k$ so that
for all $i$ if $j>i$ then $p_j$ is not in the convex hull of
$\{p_1, ..., p_i\}$.
 This is always possible due to the assumption that $p_1, ...,
p_k$ are distinct. Indeed, choose a line $L$ so that the
orthogonal projection of $p_i$'s to $L$ are distinct. Now relabel
these points according to the linear order of the projections to
$L$.

For this choice of ordering of $p_1, ..., p_k$, we construct $h_1,
..., h_k$ inductively so that $W_i(h)$ contains a non-empty open
set. Let $h_1=0$, since $p_2 \neq p_1$, for any choice of $h_2$,
both $vol(\{ x |\bigtriangledown u_{(h_1, h_2)}(x) =p_i\})>0$ for
$i=1,2$. Inductively, suppose $h_1, ..., h_i$ have been
constructed so that $vol(\{ x | \bigtriangledown u_{(h_1, ...,
h_i)}(x) = p_j\})>0$ for all $j=1,2,..., i$.  To construct
$h_{i+1}$, first note that since $p_{i+1}$ is not in the convex
hull of $p_1, ..., p_i$, by part (a), for any choice of $h_{i+1}$,
$vol(W_{i+1}(h_1, ..., h_{i+1})) >0 $ and $W_{i+1}(h_1, ...,
h_{i+1})$ is unbounded. Now by choosing $h_{i+1}$ very negative,
we can make all $vol(W_j(h_1, ..., h_{i+1}))>0$ for all
$j=1,2,..., i+1$.
\end{proof}

It is known that convex subdivisions associated to a PL convex
function $u_h(x)$ on $\bf R^n$ are exactly the same as the power
diagrams. See for instance \cite{Amgous}, \cite{Siada}. We recall
briefly the power diagrams. Suppose $P=\{ p_1, ..., p_k\}$ is a
set of $k$ points in $\bf R^n$ and $w_1, ..., w_k$ are $k$ real
numbers. The power diagram for the weighted points $\{ (p_1, w_1),
..., (p_k, w_k)\}$ is the convex subdivision $\mathcal T$ defined
as follows. The top-dimensional cells are $U_i=\{ x \in \bf R^n
$$| |x-p_i|^2 + w_i \leq |x-p_j|^2 + w_j$ for all $j$\}. Here
$|x|^2 = x \cdot x$ is the square of the Euclidean norm and
$|x-p_i|^2+w_i$ is the \it power distance \rm from $x$ to $(p_i,
w_i)$. If all weights are zero, then $\mathcal T$ is the Delaunay
decomposition associated to $P$. Since $ |x-p_i|^2 + w_i \leq
|x-p_j|^2 + w_j$ is the same as $x \cdot x -2 x \cdot p_i +
|p_i|^2 + w_i \leq x \cdot x -2 x \cdot p_j + |p_j|^2 + w_j$ which
is the same as $ x \cdot p_i -\frac{1}{2}( |p_i|^2 + w_i) \geq x
\cdot p_j -\frac{1}{2}( |p_j|^2 + w_j)$. We see that $U_i =\{ x
\in \bf R^n |$$ x \cdot p_i + h_i \geq x \cdot p_j +h_j$ for all
$j$\} where $h_i = -\frac{1}{2}(|p_i|^2+w_i)$. This shows the well known
fact that, 

\begin{proposition} The power diagram associated to $\{(p_i,
w_i)| i=1, ..., k\}$ is the convex subdivision associated to the
PL convex function $u_h$ defined by (\ref{eq:PL}) where $h_i =
-\frac{ |p_i|^2+w_i}{2}$.
\label{prop:powerdiag}
\end{proposition}

\begin{figure}[!t]
\centering
\includegraphics[width=0.5\textwidth]{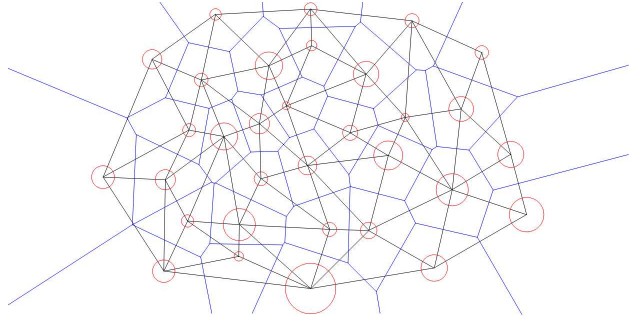}
\caption{Power diagram and its dual weighted Delaunay triangulation.
\label{fig:power_diagram}}
\end{figure}

\subsection{Variation of the volume of top-dimensional cells}

The following is the key technical proposition for us to establish
variational principles.

\begin{proposition} \label{prop:14} Suppose $\sigma: \Omega \to \bf R$ is
continuous defined on a compact convex domain $\Omega \subset \bf
R^n$. If $p_1, ..., p_k \in \bf R^n$ are distinct and $h \in \bf
R^k$ so that $vol( W_i(h) \cap \Omega) >0$ for all $i$, then
$w_i(h) =\int_{W_i(h) \cap \Omega} \sigma (x) dx$ is a
differentiable function in $h$ so that  for $j \neq i$ and $W_i(h)
\cap \Omega$ and $W_j(h) \cap \Omega$ share a codimension-1 face
$F$, 
\begin{equation}
\label{eq:hessian} 
\frac{\partial w_i(h)}{\partial h_j} =-\frac{1}{|p_i -p_j|} \int_{F} \sigma|_F(x)dA 
\end{equation} 
where $dA$ is the area form on $F$ and partial
derivative is zero otherwise.  In particular,
$$\frac{\partial w_i(h)}{\partial h_j} =\frac{\partial
w_j(h)}{\partial h_i}.$$

\end{proposition}
\begin{proof}

The proof is based on the following simple lemma.

\begin{lemma} \label{lemma:2} Suppose $X$ is a compact domain in $\bf R^n$,
$f: X \to \bf R$ is a non-negative continuous function and
$\tau(x,t):\{ (x,t) \in X \times \bf R |$$ 0 \leq t \leq f(x)\}
\to \bf R$ is continuous. For each $t \geq 0$, let $f_t(x)
=\min\{t, f(x)\}$. Then $W(t) =\int_X (\int_0^{f_t(x)} \tau(x,s)
ds) dx$ satisfies
\begin{equation}\label{eq:12} \lim_{ t \to t^+_0} \frac{ W(t) -W(t_0)}{t
-t_0} = \int_{\{ x | f(x) > t_0\}} \tau(x, t_0) dx \end{equation}
and \begin{equation} \label{eq:13} \lim_{ t \to t^-_0} \frac{ W(t)
-W(t_0)}{t -t_0}  = \int_{\{ x | f(x) \geq t_0\}} \tau(x, t_0) dx.
\end{equation}
It is  differentiable at $t_0$ if and only if $\int_{\{x \in X |
f(x)=t_0\}} \tau(x, t_0) dx =0$.
\end{lemma}

\begin{figure}[!t]
\centering
\includegraphics[width=0.7\textwidth]{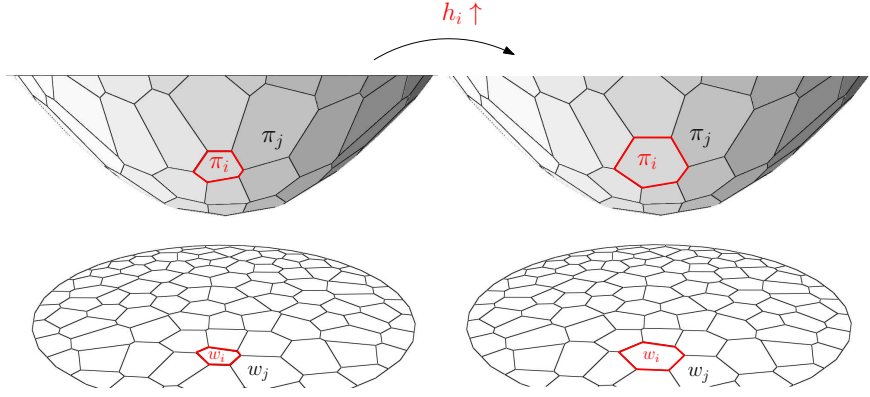}
\caption{Variation of the volume of top-dimensional cells¡£
\label{fig:variation_1}}
\end{figure}

 The conditions in the lemma can be relaxed.

\begin{proof}  Let $G_t(x) =\int_0^{f_t(x)} \tau(x, s) ds$  and $M$ be an upper
bound of $|\tau(x, t)|$ in its domain. Since $|\min(a,b)
-\min(a,c)| \leq |b-c|$, we have $|f_t(x) - f_{t'}(x)| \leq
|t-t'|$. Now, for any $t \neq t'$,

\begin{equation}\label{eq:123}
|\frac{ G_t(x)-G_{t'}(x)}{t-t'} |=\frac{1}{|t- t'|}
|\int_{f_{t'}(x)}^{f_t(x)} \sigma(x,s) ds|  \leq \frac{M}{|t-t'|}
| f_t(x) -f_{t'}(x)| \leq M.
\end{equation}

Fix $t_0$ and $x \in X$. If $f(x) < t_0$, then for $t$ very close
to $t_0$, $G_t(x) = \int^{f(x)}_0 \tau(x, s) ds$. Hence $ \lim_{t
\to t_0} \frac{ G_t(x)-G_{t_0}(x)}{t-t_0} =0$. If $f(x)
> t_0$, then for $t$ very close to $t_0$, $G_t(x) = \int^{t}_0
\tau(x, s) ds$.  Hence $ \lim_{t \to t_0} \frac{
G_t(x)-G_{t_0}(x)}{t-t_0} = \lim_{t \to t_0} \frac{1}{t-t_0}
\int_{t_0}^t  \tau(x, s) ds =\tau(x, t_0)$. If $f(x) = t_0$, then
the above calculations shows $\lim_{t \to t_0^+} \frac{
G_t(x)-G_{t_0}(x)}{t-t_0} =0$ and $ \lim_{t \to t_0^-} \frac{
G_t(x)-G_{t_0}(x)}{t-t_0} = \tau(x, t_0)$. Therefore, by Lebesgue
dominated convergence theorem, we have
\begin{equation} \lim_{ t \to t^+_0} \frac{ W(t) -W(t_0)}{t
-t_0} =\lim_{ t \to t_0^+} \int_{X} \frac{ G_t(x)
-G_{t_0}(x)}{t-t_0}  dx = \int_{\{ x | f(x) > t_0\}} \tau(x, t_0)
dx
\end{equation} and
\begin{equation} \lim_{ t \to t^-_0} \frac{
W(t) -W(t_0)}{t -t_0} =\lim_{ t \to t_0^-} \int_{X} \frac{ G_t(x)
-G_{t_0}(x)}{t-t_0}  dx = \int_{\{ x | f(x) \geq t_0\}} \tau(x,
t_0) dx \end{equation}
 This establishes the lemma.

\end{proof}

Fix $a < b$, we call the domain $\{(x, t) \in X \times \bf R$$|  a
\leq f(x), a \leq t \leq \min(f(x),b) \}$ a \it cap domain \rm
with base $\{x | f(x) \geq a\}$ and top $\{ x | f(x) \geq b\}$ of
\it height \rm $(b-a)$ associated to the function $f$.

To prove the proposition \ref{prop:14}, let $h'=(h_1, ...,
h_{i-1}, h_i -\delta, h_{i+1}, ..., h_k)$. For small positive
$\delta >0$, by definition, $W_i(h') \subset W_i(h)$ and $W_j(h)
\subset W_j(h')$. If $W_i(h) \cap W_j(h) \cap \Omega =\emptyset$,
then $W_j(h)\cap \Omega = W_j(h') \cap \Omega$ for small $\delta$.
Hence $\frac{\partial W_j(h)}{\partial h_i}=0$. If $W_i(h) \cap
\Omega$ and $W_j(h) \cap \Omega$ share a codimension-1 face $F$,
then the closure $cl(W_j(h') -W_j(h))$ is a cap domain with base
$F$ associated to a convex function $f$ defined on $F$. The height
of the cap domain is $\frac{1}{|p_i -p_j|} \delta$ and $f$ is PL
convex so that the $(n-1)$-dimensional Lebesgue measure of set of
the form $\{ x \in F | f(x) =t\}$ is zero. Furthermore for
$\delta>0$, by definition
$$ \frac{ w_j(h') -w_j(h)}{\delta} = \frac{1}{\delta}
\int_{W_j(h')\cap \Omega -W_j(h) \cap \Omega} \sigma(x) dx$$
$$=\frac{1}{\delta} \int_F \int_{0}^{f_t(y)} \tau(y, s) ds dy$$
where $y \in F$ is the Euclidean coordinate and $\tau(y,s)$ is
$\sigma$ expressed in the new coordinate. Thus, by lemma
\ref{lemma:2}, we see \begin{equation}\label{eq:17}\lim_{\delta
\to 0^+} \frac{ w_j(h') -w_j(h)}{\delta} =\int_F \sigma|_F dA.
\end{equation} The same calculation shows that for $\delta <0$ and close to
0, using the fact that $cl(W_j(h)\cap \Omega -W_j(h')\cap \Omega)$
is cap with top $F$, we see that (\ref{eq:17}) holds as well.
Finally, if $W_i(h) \cap \Omega$ and $W_j(h)\cap \Omega$ share a
face of dimension at most $n-2$, then the same calculation still
works where the associate cap domain has either zero top area or
zero bottom area. Thus the result holds.

\end{proof}

\section{A proof of Theorem~\ref{thm:otp}}

Our proof is divided into several steps. In the first step, we
show that the set $H=\{ h \in \bf R^k$$| vol(W_i(h) \cap \Omega)
>0$ for all $i$\} is a non-empty open convex set. In the second step, we
show that $E(h) =\int_{\Omega} u_h(x) \sigma(x) dx$ is a
$C^1$-smooth convex function on $H$ so that $\frac{\partial
E(h)}{\partial h_i} =\int_{W_i(h) \cap \Omega} \sigma(x) dx$. In
the third step, we show that $E(h)$ is strictly convex in $H_0= H
\cap \{ h | \sum_{i=1}^k h_i =0\}$. In the fourth step, we show
that the gradient map $\bigtriangledown E: H_0 \to W =\{ (A_1,
..., A_k) | A_i>0, \sum_{i=1}^k A_i = \int_{\Omega} \sigma(x)
dx$\} is a diffeomorphism. Finally, for the completeness, 
we include a  simple proof by Aurenhammer et al. 
(Lemma 1 in \cite{aure1}) to show that $\nabla u_b$ is an optimal 
transport map minimizing the quadratic cost. 

\subsection{Convexity of the domain $H$}

We begin with a simple observation that a compact convex set $X
\subset \bf R^n$ has positive volume if and only if $X$ contains a
non-empty open set, i.e., $X$ is n-dimensional. Therefore,
$vol(W_i(h) \cap \Omega)
>0$ is the same as $W_i(h) \cap \Omega$ contains a non-empty open
set in $\bf R^n$. The last condition,  by the above proposition
\ref{prop:21}(a), is the same as there exists $x \in \Omega$ so
that $x \cdot p_i > \max_{j \neq i}\{ x \cdot p_j + h_j\}.$

Now to see that $H$ is convex, since $H =\cap_{i=1}^k H_i$ where
$H_i=\{ h \in \bf R^k$$ | vol(W_i(h) \cap \Omega) >0\}$, it
suffices to show that $H_i$ is convex for each $i$. To this end,
take $\alpha, \beta \in H_i$ and $t \in (0, 1)$. Then there exist
two vectors $v_1, v_2 \in \Omega$ so that $ v_1 \cdot p_i +
\alpha_i > v_1 \cdot p_j + \alpha_j$ and $v_2 \cdot p_i+\beta_i >
v_2 \cdot p_j + \beta_j$ for all $j \neq i$. Therefore,
$(tv_1+(1-t)v_2)\cdot p_i + (t \alpha_i + (1-t) \beta_i) > (tv
_1+(1-t)v_2) \cdot p_j + (t\alpha_j + (1-t)\beta_j)$ for all $j
\neq i$. This shows that $t\alpha+(1-t)\beta$ is in $H_i$.
Furthermore, each $H_i$ is non-empty. Indeed, given $h_1, ...,
h_{i-1}, h_{i+1}, ..., h_k$, by taking $h_i$ very large, we see
that $h=(h_1,..., h_k)$ is in $H_i$.  Also, from the definition,
$H_i$ is an open set. Therefore, to show $H$ is an open convex
set, it remains to show that $H$ is non-empty.

To see $H \neq \emptyset$, it suffices to show that there exists
$h$ so that $vol(W_i(h)) >0$ (which could be $\infty$) for all
$i$. Indeed, after some translation, we may assume that $0$ is in
the interior of $\Omega$. Due to the fact that $u_{\lambda h}(x)
=\lambda u_h(x/\lambda)$ and $W_i(\lambda h)=W_i(h)/\lambda$,  if
$vol(W_i(h))>0$ for all $i$, then for $\lambda >0$ large,
$vol(W_i(\lambda h) \cap \Omega)>0$. Now by proposition
\ref{prop:21}(e), we can find $h$ so that $vol(W_i(h))>0$ for all
$i$. Thus $H \neq \emptyset$.

\subsection{Convexity of the function $E(h)$ and its gradient}
We show that $E(h) = \int_{\Omega} u_h(x) \sigma(x) dx$ is convex in $h$ 
and satisfies $\partial E(h)/\partial h_i = \int_{W_i(h)\cap \Omega} \sigma(x) dx$. 
Since $u_h(x)$ is the maximum of a collection of functions $x
\cdot p_i + h_i$ which are linear in $(x, h) \in \bf R^n \times
\bf R^k$, $u_h(x)$ is convex in $(x, h) \in \bf R^n \times \bf
R^k$. Thus by the assumption that $\sigma(x) \geq 0$,
$ \int_{\Omega} u_h(x) \sigma(x) dx$ is convex in $h$.
For a convex function $f$, the directional derivative of $f$ at
the point $\alpha$ in direction $d$, denoted by $f'(\alpha, d)
=\lim_{ t \to 0^+} \frac{f(\alpha + t d) -f(\alpha)}{t}$, always
exists in $(0, \infty]$. Furthermore, it is known (proposition
2.3.2 in \cite{brow}) that if  $f_i$ are convex and $f(x)
=\max_{i=1}^k\{ f_i(x)\}$, then $f'(\alpha, d) = \max_{ i \in
K_{\alpha}} \{ f'_i(\alpha, d)\}$ where $K_{\alpha}=\{ j |
f_j(\alpha) =f(\alpha)\}$.  In our case, $x \cdot p_i + h_i$ is
smooth so that $\frac{\partial (x \cdot p_i + h_i)}{\partial h_j}
=\delta_{ij}$. Therefore, we obtain
$$\frac{\partial u_h(x)}{\partial h_j} =\max_{ i \in K_{x, h}} \{
\delta_{ij}\}$$ where $K_{x, h} =\{ m | x\cdot p_m + h_m
=u_h(x)\}$. By proposition \ref{prop:21}(a), if $x \in
int(W_j(h))$, then $\frac{\partial u_h(x)}{\partial h_j} =1$ and
if $x \in int(W_i(h))$  for $i \neq j$, then $\frac{\partial
u_h(x)}{\partial h_j}=0$. In particular, this shows $|u_h(x)
-u_{h'}(x)| \leq |h -h'|$.

By the dominant convergence theorem in real analysis and
continuity of $\sigma(x)$, we obtain
$$ \frac{\partial E(h)}{\partial h_i} = \int_{\Omega}
\frac{\partial u_h(x)}{\partial h_i} \sigma(x) dx =\int_{W_i(h)
\cap \Omega} \sigma(x) dx.$$

This shows that $E(h) =\int^h \sum_i \int_{W_i(h)\cap \Omega} \sigma(x)dx dh_i +c$ for some constant $c$.

\subsection{The strict convexity of $E(h)$}

We will show that the Hessian matrix of $E(h)$ has a 1-dimensional
null space spanned by the vector $(1, 1, ..., 1)$. Let $w_i(h)
=\int_{W_i(h) \cap \Omega} \sigma(x) dx$. By the calculation
above, $\partial E(h)/\partial h_i =w_i(h)$ and $\sum_{i=1}^k
w_i(h) =\int_{\Omega} \sigma(x) da$. By proposition \ref{prop:14},
for $i \neq j$, $\frac{\partial W_i(h)}{\partial h_j} \leq 0$.
Furthermore, if $W_i(h)$ and $W_j(h)$ share a codimension-1 face
$F$ in $\Omega$, then $\frac{\partial W_i(h)}{\partial h_j} =
-\frac{1}{|p_i-p_j|} \int_{F} \sigma|_F(x) dA <0$ where
$dA$ is the area form on the codimension-1 face $F$. This implies,

\begin{corollary}\label{coro:21}  The Hessian matrix $Hess(E)$ of $E(h)$ is
positive semi-definite with 1-dimensional null space generated by
$(1,1,.., 1)$.  In particular, $E|_{H_0}: H_0=\{h \in H |
\sum_{i=1}^k h_i =0\} \to \bf R$ is strictly convex.
\end{corollary}

\begin{proof} By proposition \ref{prop:14} and the fact that $\sum_{i=1}^k
\partial w_i/\partial h_j = \partial (vol(\Omega)) /\partial h_j=0$,
it follows that the Hessian matrix $Hess(E)=[\frac{\partial
w_i(h)}{\partial h_j}] =[a_{ij}] $ is diagonally dominated  (i.e.,
all diagonal entries are positive and all off diagonal entries are
non-positive so that the sum of entries of each row is zero).
Therefore, it is positive semi-definite with kernel containing the
vector $(1,1,..., 1)$. To see that the kernel is 1-dimensional,
suppose $v=(v_1, ..., v_k)$ is a non-zero vector so that $Hess(E)
v^t=0$ ($v^t$ is the transpose of $v$). Let us assume without loss
of generality that $|v_{i_1}| = \max_{i} \{ |v_i|\}$ and $v_{i_1}
>0$. In this case, using $a_{i_1 i_1}v_{i_1} =\sum_{j \neq i_1}
a_{i_1j} v_j$ and $a_{i_1 i_1} =-\sum_{j \neq i_1} |a_{i_1 j}|$,
we see that $v_j=v_{i_1}$ for all indices $i$ with $a_{i_1 j} \neq
0$. It follows that index set $I= \{ i | v_i =\max_{j}
\{|v_j|\}\}$ has the following property. If $i_1 \in I$ and
$a_{i_1 i_2} \neq 0$, then $i_2 \in I$. We claim that $I=\{
1,2,..., k\}$, i.e., $v =v_1(1,1,...,1)$. Indeed, for any two
indices $i \neq j$, since $\Omega$ is connected and $W_r(h)$
  are convex,
 there exists a sequence of
indices $i_1=i, i_2, ..., i_m =j$ so that $W_{i_{s}}(h) \cap
\Omega$ and $W_{i_{s+1}}(h) \cap \Omega$ share a codimension-1
face for each $s$. Therefore $a_{i_s i_{s+1}} \neq 0$. Translating
this to the Hessian matrix $Hess(E) =[a_{ij}]$, it says that for
any two diagonal entries $a_{ii}$ and $a_{jj}$, there exists a
sequence of indices $i_1=i, i_2, ..., i_m=j$ so that $a_{i_s
i_{s+1}} <0$. This last condition together with the property of
$I$ imply $I=\{1,2,..., k\}$, i.e., $dim(Ker(Hess(E))) =0$.

\end{proof}

\subsection{$\nabla E$ is a local diffeomorphism}
With above preparations, we can show the gradient map $\nabla E$
is a diffeomorphism. Let $\Phi = \nabla E: H_0 \to W=\{ (A_1, ...,
A_k) \in \bf R^k$$|A_i > 0, \sum_{i=1}^k A_i =\int_{\Omega}
\sigma(x) dx$\} be the map sending $h$ to $(w_1(h), ..., w_k(h))$.
By the calculation above, $\Phi(h)$ is the gradient of $E|_{H_0}$ at
$h$. Since $E|_{H_0}$ has positive definite Hessian matrix in
$H_0$, its gradient $\Phi$ is an injective local diffeomorphism from
$H_0$ to $W$. But $dim(H_0)=dim(W)$, thus $\Phi(H_0)$ is open in $W$.
Thus to finish the proof that $\Phi(H_0)=W$, using the fact that $W$
is connected, we only need to show that $\Phi(H_0)$ is closed in $W$.
To see this, take a sequence of point $h^{(m)}$ in $H_0$ so that
$\Phi(h^{(m)})$ converges to a point $\alpha \in W$. We claim that
$\alpha \in \Phi(H_0))$. First note that $h^{(m)}$'s are bounded in
$\bf R^k$. Indeed, if not, we can choose a convergent subsequence,
still denoted by $h^{(m)}$, so that there are two indices  $i \neq
j$ with $h^{(m)}_i \to \infty$ and $h^{(m)}_j \to -\infty$. Since
$\Omega$ is compact, we see for $m$ large $x \cdot p_i + h_i^{(m)}
> x \cdot p_j + h^{(m)}_j$ for all $x \in \Omega$. This shows that
$W_j(h^{(m)}) \cap \Omega =\emptyset$ for $m$ large which
contradicts the assumption that $vol(W_j(h^{(m)}) \cap \Omega) >0$
for $m$ large. As a consequence, we can choose a convergence
subsequence, still denoted by $h^{(m)} \to h \in \bf R^k$. For
this $h$, by the continuity of the map $h \to (w_1(h), ...,
w_k(h)$ on $\bf R^k$, we have $\Phi(h) =\alpha$, i.e., $h \in H$ and
$\alpha \in \Phi(H_0)$.

As a consequence of the proof, we also obtained a new proof of the
infinitesimal rigidity theorem of Alexandrov.

\begin{corollary} (Alexandrov)  The map $\bigtriangledown E : H_0 \to W$
sending the normalized heights $h$ to the area vector $(w_1(h),
..., w_k(h))$ is a local diffeomorphism.
\label{cor:inft}
\end{corollary}

\subsection{ $\nabla u_b$ is an optimal transport map}
We reproduce the elegant proof by Aurenhammer et al \cite{aure1} 
here for completeness.
Notice the quadratic transport cost of $\nabla u_b$ is 
$\sum_{i=1}^k \int_{W_i} |p_i - x|^2 \sigma(x) dx$. From Proposition~\ref{prop:powerdiag}, $\{W_1, ..., W_k\}$ 
is the power diagram associated 
to $\{(p_i, w_i = -|p_i|^2 - 2b_i)\}$ in $R^d$. Suppose $\{ U_1, ..., U_k \}$ is any 
partition of $R^d$ so that $\int_{U_i} \sigma(x) dx =\int_{W_i} \sigma(x) dx, i = 1, 2, \cdots, k$. 
By the definition of the power diagram, we have 
$$\sum_{i=1}^k \int_{W_i} ( | x-p_i|^2 + w_i) \sigma(x) dx  \leq  \sum_{i=1}^k \int_{U_i} (|x-p_i|^2 + w_i) \sigma(x) dx.$$
Thus  
$$\sum_{i=1}^k \int_{W_i}  | x-p_i|^2  \sigma(x) dx )  \leq  \sum_{i=1}^k \int_{U_i} |x-p_i|^2 \sigma(x) dx. $$
This shows that $\nabla u_b$ minimizes the quadratic transport cost.

\section{A proof of Theorem \ref{thm:dmae}}

We fix $g_1, ..., g_m$ through out the proof. For simplicity, let
$p_{k+j} = v_j$ and $h_{k+j}=-g_j$ for $j=1,..., m$ and let
$$ W_i(h) =\{ x \in {\bf R^n}| x \cdot p_i +h_i \geq x \cdot p_j+h_j, j=1,..., k+m\}.$$
Define $$H =\{ h \in {\bf R^k} | vol(W_i(h)) >0, i=1,..., k+m\}.$$

\begin{lemma} (a) H is a non-empty open convex set in $\bf R^k$.

(b) For each $h \in H$ and $i=1,..., k$ and $j=1,..., m$, $W_i(h)$
is a non-empty bounded convex set and $W_{k+j}(h)$ is a non-empty
unbounded set.

\end{lemma}

\begin{proof}  The proof of convexity of $H$ is exactly the same
as that of \S 2.1. We omit the details. Also, by definition $H$ is
open.  To show that $H$ is non-empty, using proposition
\ref{prop:21}(e), there exists $\bar h \in \bf R^k$ so that for
all $i=1,..., k$, $vol( W_i(\bar h)) >0$. We claim for $t>0$ large
the vector $h=\bar h+(t,..., t) \in H$. Indeed, let $B$ be a large
compact ball so that $B \cap W_i(\bar h) \neq \emptyset$ for all
$i=1,..., k$. Now choose $t$ large so that
$$ \min_{ x \in B} \{ x \cdot p_i + h_i | i=1,2,...,k\} > \max_{ x
\in B}\{ x \cdot v_j + g_j |j=1,..., m\}.$$ For this choice of
$h$, by definition, $W_i(\bar h) \cap B \subset W_i(h)$.

Part (b) follows from proposition \ref{prop:21} (b) and (c).

\end{proof}

For $h \in H$ and $i=1,..., k$, let $w_i(h) =vol(W_i(h))>0$. For
each $h \in H$, by proposition \ref{prop:14} applied to a large
compact domain $X$ whose interior contains $\cup_{i=1}^k W_i(h)$,
we see that $w_i(h)$ is a differentiable function so that
$\frac{\partial w_i}{\partial h_j} =\frac{\partial w_j}{\partial
h_i}$ for all $i, j=1, ..., k$. Thus the differential 1-form $\eta
=\sum_{i=1}^k w_i(h) dh_i$ is a closed 1-form on the open convex
set $H$. Since $H$ is simply connected, there exists a
$C^1$-smooth function $E(h): H \to \bf R$ so that $\frac{\partial
E}{\partial h_i} =w_i(h)$.

\begin{lemma}\label{lemma:234}  The Hessian matrix $Hess(E)$ of $E$ is positive
definite for each $h \in H$. In particular, $E$ is strictly convex
and $\bigtriangledown E: H \to \bf R^k$ is a smooth embedding.
\end{lemma}
\begin{proof}
By the same proof as in \S2.3, we have for $i \neq j$, $\partial
w_i(h)/\partial h_j = -\frac{1}{|p_i-p_j|} Area(F)<0$
if $W_i(h)$ and $W_j(h)$ share a codimension-1 face $F$ and it is
zero otherwise. Furthermore, for each $j=1,...,k$, $\sum_{i=1}^k
\partial w_i/\partial h_j = \frac{\partial ( \sum_{i=1}^k
w_i(h))}{\partial h_j} >0$ if $W_j(h)$ and one of $W_{\mu+k}(h)$
share a codimension-1 face. It is zero otherwise. This shows the
Hessian matrix $Hess(E)=[a_{ij}]$ is diagonally dominated so that
$a_{ij} \leq 0$ for all $ i \neq j$ and $a_{ii} \geq \sum_{ j \neq
i} |a_{ij}|$. Thus $Hess(E)$ is positive semi-definite. To show
that it has no kernel, we proceed with the same argument as in the
proof of corollary \ref{coro:21}. The same argument shows that if
$b=(b_1, ..., b_k)$ is a null vector for $[a_{ij}]$, then
$b_1=b_2=...=b_k$. On the other hand, there is an index $i$ so
that $W_i(h)$ and one of $W_{j+k}(h)$ share a codimension-1 face,
i.e, $a_{ii} > \sum_{i=1}^k |a_{ij}|$. Using $\sum_{j=1}^k
a_{ij}b_1=0$, we see that $b_1=0$, i.e., $b=0$. This establishes
the lemma.

\end{proof}

Now we prove theorem \ref{thm:dmae} as follows. Let $A =\{(A_1,
..., A_k) | A_i >0\}$ and $\Phi=\bigtriangledown E: H \to A$ be the
gradient map. By lemma \ref{lemma:234}, $\Phi$ is an injective local
diffeomorphism from $H$ to $A$. In particular, due to
$dim(H)=dim(A)$,  $\Phi(H)$ is open in $A$. To finish the proof that
$\Phi(H)=A$, since $A$ is connected, it suffices to prove that $\Phi(H)$
is closed in $A$. Take a sequence of points $h^{(i)}$ in $H$ so
that $\Phi(h^{(i)})$ converges to a point $a \in A$. We claim that $a
\in \Phi(H)$. After taking a subsequence, we may assume that
$h^{(i)}$ converges to a point in $[-\infty, \infty]^k$. We first
show that $\{h^{(i)}\}$ is a bounded set in $\bf R^k$. If
otherwise, there are three possibilities: (a) there is $j$ so that
$h^{(i)}_j \to -\infty$ as $i \to \infty$, (b) there are two
indices $j_1$ and $j_2$ so that $\lim_{i \to \infty} h^{(i)}_{j_1}
=\infty$ and \{$h^{(i)}_{j_2}$\} is bounded, and (c) for all
indices $j$, $\lim_{i \to \infty} h^{(i)}_j =\infty$.  In the
first case (a), due to $p_j \in int(conv(v_1, ..., v_m))$ and
$h^{(i)}_j$ is very negative, $x \cdot p_j + h^{(i)}_j < \max\{x
\cdot v_{j'} + g_{j'}| j'=1,..., m\}$ for $i$ large for all $x$.
This implies for $i$ large $W_j(h^{(i)}) =\emptyset$ which
contradicts the assumption that $\lim_{i} \Phi(h^{(i)}) =a \in A$. In
the case (b) that \{$h^{(i)}_{j_2}$\} is bounded, then the sets
$W_{j_2}(h^{(i)})$ lies in a compact set $B$. For $i$ large,  $ x
\cdot p_{j_1} + h^{(i)}_{j_1} \geq \max\{ x \cdot v_j+g_j, x \cdot
p_{j_2}+h^{(i)}_{j_2})\}$ for all $x \in B$. This implies that
$W_{j_2}(h^{(i)}) =\emptyset$ for large $i$ which contradicts the
assumption that $\lim_{i} \Phi(h^{(i)}) =a \in A$. In the last case
(c), since for each $j$, $\lim_{i \to \infty} h^{(i)}_j =\infty$,
for any compact set $B$, there is an index $i$ so that $B \subset
\cup_{ \mu=1}^k W_{\mu}(h^{(i)})$. This implies that the sum of
the volumes $\sum_{\mu=1}^k vol(W_{\mu}(h^{(i)}) $ tends to
infinity which again contradicts the assumption $\lim_{i}
\Phi(h^{(i)})=a \in A$.

Now that $h^{(i)}$ is convergent to a point $h$ in $\bf R^k$, by
the continuity of the map sending $h$ to $(w_1(h), ..., w_k(h))$
on $\bf R^k$, we see that $\Phi(h) =a$. This shows $h \in H$ and $a
\in \Phi(H)$, i.e., $\Phi(H)$ is closed in $A$.

Hence, given any $(A_1, ..., A_k) \in A$, there exists a unique $h
\in H$ so that $\Phi(h) =(A_1, ..., A_k)$. Let $u =\max\{ x \cdot p_i
+h_i | i=1,..., k+m\}$ be the PL convex function on $\bf R^n$ and
$w$ be its dual. By corollary 2.1, we conclude that the vertices
of $w$ are exactly $\{v_i, p_j | i, j\}$ with $w(v_i) =g_i$ and
$w(p_j)=-h_j$ so that the discrete Hessian of $w$ at $p_i$, which
is $w_i(h)=A_i$. Furthermore, by proposition \ref{prop:21}, the
associated convex subdivision of $w$ on $\Omega$ has exactly the
vertex set $\{v_1, ..., v_k, p_1, ..., p_k\}$.

\bibliographystyle{plain}
\bibliography{OTP_MAE}
\end{document}